\documentclass{amsart}
\usepackage{amsmath}
\usepackage{amssymb}
  \usepackage{paralist}
  \usepackage{graphics} 
  \usepackage{epsfig} 
 \usepackage[colorlinks=true]{hyperref}
\hypersetup{urlcolor=blue, citecolor=red}

\newtheorem{theorem}{Theorem}[section]
\newtheorem{corollary}[theorem]{Corollary}

\newtheorem{lemma}[theorem]{Lemma}
\newtheorem{proposition}[theorem]{Proposition}

\theoremstyle{definition}

\newcommand{\ep}{\varepsilon}

\title[Entropy and Ergodic Measures for Toral Automorphisms]
      {Entropy and Ergodic Measures for Toral Automorphisms}

\author[Peng Sun]{}


 \email{pengsunmath@gmail.com}

\begin{document}

\maketitle

\centerline{\scshape Peng Sun }
\medskip
{\footnotesize
 \centerline{China Economics and Management Academy}
   \centerline{Central University of Finance and Economics}
   \centerline{Beijing 100081, China}
} 

\bigskip

\begin{abstract}
    In this paper we show that for every linear toral automorphism,
    especially the non-hyperbolic ones,
    the entropies of ergodic measures form a dense set on the interval 
    from zero to the topological entropy.
\end{abstract}


\section{Introduction}

Let $f$ be a diffeomorphism on a compact manifold. $f$ has some
ergodic invariant measures. By Variational Principle, the supremum of their
metric entropies equals the topological entropy of $f$. But only few results have been obtained about the structure of the set
$$E(f)=\{h_\mu(f)| \mu\text{ is an ergodic invariant measure of }f\}.$$
Katok conjectures $E(f)\supset [0, h(f))$ if $f$ has enough regularity.
He proved that this is the case for $C^r (r>1)$ diffeomorphisms on surfaces \cite{Ka83}.
Recently we extended Katok's result to certain skew product cases \cite{Su10}.

In this paper we deal with linear
toral automorphisms, especially the non-hyperbolic ones.
Precisely, let $f_A$ be a linear toral automorphism
induced by $A\in GL(2,\mathbb{Z})$. We have

\begin{theorem}
For any linear toral automorphism $f_A$, $E(f_A)$ is dense in $[0, h(f_A)]$.
\end{theorem}

In fact, we believe that $E(f)\supset[0, h(f))$. However,
we are not able to carry through it at the moment. The critical obstacle
is the discontinuity of entropy map on invariant measures.

To prove the theorem, we only need to show the following proposition. 
When $f_A$ is reducible, it can be
 decomposed into a
direct product of irreducible automorphisms.

\begin{proposition}\label{promain}
Let $f_A$ be an irreducible linear automorphism of the torus. For every
$\beta_1,\beta_2\in[0, h(f)]$ and $\beta_1<\beta_2$, there is a compact invariant
set $M_{\beta_1,\beta_2}$ such that $\beta_1\le h(f_A|_{M_{\beta_1,\beta_2}})\le\beta_2$.
\end{proposition}

As $f_A$ is $C^\infty$, the map $\mu\to h_\mu(f_A)$ is upper-semicontinuous
on the set of all $f_A$-invariant measures \cite{New89}.
So every such $M_{\beta_1,\beta_2}$ supports a measure of
maximal entropy. We have that $E(f_A)\supset\{h(f_A|_{M_{\beta_1,\beta_2}})\}$
and hence $E(f_A)$ is dense in $[0, h(f_A)]$.

As we know, when $A$ is a hyperbolic
matrix, i.e.
none of the eigenvalues lies on the unit circle, $f_A$ is Anosov (uniformly
hyperbolic) and $E(f_A)=[0, h(f_A)]$. The result of this paper makes sense
when $f_A$ is non-hyperbolic.
Irreducible non-hyperbolic toral automorphisms have drawn some interests
and similar
problems have been considered. The basic geometric structure and dynamical
properties of these maps has been discussed in \cite{Lind82}. In \cite{LK04}
and \cite{LK05} some properties that holds for all ergodic measures are shown
using harmonic analysis and algebraic methods. For the entropy of ergodic
measures, the case that $A$ is a companion matrix of a Salem number has been
discussed, where the stable foliation is one-dimensional and there is 
no repeated complex eigenvalues. Another notable
result from \cite{LK05} is that non-hyperbolic automorphisms
do not have Markov partitions. So the way of constructing ergodic measures
from symbolic spaces, as for hyperbolic automorphisms, does not work
for non-hyperbolic ones. In this paper, we use a geometrical argument
due to Ma\~n\'e that works mainly with unstable leaves to achieve our
result.






\section{Basic Facts}
Let $X=\mathbb{T}^l$. We fix $f=f_A:X\to X$ and assume that $A$ is irreducible.
Denote $[y]=y+\mathbb{Z}^l$, the image under the covering map $\mathbb{R}^l\to\mathbb{T}^l$.
Let $\chi_1>\chi_2>\cdots>\chi_u>1$ be the absolute values of the eigenvalues
of $A$ larger than $1$, each with multiplicity $\zeta_i$. Denote $\zeta=\sum_{1\le
i\le u}\zeta_i$ and $\zeta_{u+1}=l-\zeta$.
Let $W_i$ be the $i$-th unstable foliation corresponding to $\chi_i$
and $W$ be the whole unstable foliation. 

For every $\ep>0$, there is $Q_\ep\in GL(l, \mathbb{R})$ such that
$A_\ep=Q_\ep^{-1} AQ_\ep$ is diagonal, or every Jordan block has the
 form like
$\begin{pmatrix}\chi & \ep & \\
 & \chi & \ep\\
 & & \ddots & \ddots\\
 & & & \chi & \ep\\
 & & & & \chi
\end{pmatrix}$, where $\chi$ denotes a real eigenvalue, or a $2\times 2$ matrix
corresponding to a complex eigenvalue, in which case the $\ep$'s denote matrices
whose norms are less than $\ep$. 

Let $R=\mathbb{R}^\zeta\times\{0\}\subset\mathbb{R}^l$.
Define $$\Phi_x=\Phi_{x,\ep}: R\to W(x),
\Phi(y)=[x+Q_\ep\cdot y].$$ $\Phi_x$ is a diffeomorphism.


For $r_i\ge 0, i=1,2,\dots, u$, let $\bar r=(r_1, r_2, \dots, r_u)$, and
\begin{align*}
D(\bar r)=\{&(x_1, x_2, \dots, x_u,0)\in R|
x_i\in\mathbb{R}^{\zeta_i}, \|x_i\|\le r_i, i=1, 2, \dots, u\}
\end{align*}

\begin{lemma}
If $\ep<\chi_j, j=1,2,\dots,u$, then
$$A_\ep(D(r_{1},\dots, r_{u}))\supset D
((\chi_{1}-\ep)r_{1}, \dots, (\chi_{u}-\ep)r_u),$$
$$A_\ep(D(r_{1},\dots, r_{u}))\subset D
((\chi_{1}+\ep)r_{1}, \dots, (\chi_{u}+\ep)r_{u}).$$
\end{lemma}

\begin{corollary}\label{excor}
If $\ep<\chi_j, j=1,2,\dots,u$, then for every $k$,
$$f^k(\Phi_x(D(r_{1},\dots, r_{u})))\supset \Phi_{f^k(x)}(D
((\chi_{1}-\ep)^kr_{1}, \dots, (\chi_{u}-\ep)^kr_{u})),$$
$$f^k(\Phi_x(D(r_{1},\dots, r_{u})))\subset \Phi_{f^k(x)}(D
((\chi_{1}+\ep)^kr_{1}, \dots, (\chi_{u}+\ep)^kr_{u})).$$

\end{corollary}

Let $U=U_{\bar r, \ep}=\pi_\zeta(Q_\ep(D(\bar
r)))\subset(-\frac12,\frac12)^\zeta$,
where $\pi_\zeta:\mathbb{R}^l\to\mathbb{R}^\zeta$ is the projection to the
first $\zeta$ coordinates.
Consider
$$K(\bar r)=K(\bar r,\ep)=\{[y]| y\in\mathbb{R}^l, \pi_\zeta(y)\in U\}.$$
Unless $\ep$ is specified and has to be taken into consideration, it is usually
omitted for convenience.
$K(\bar r)$ is compact. For every $x$,
$\Phi_x^{-1}(K(\bar r)\cap W(x))$ is the union
of infinitely many copies of $D(\bar r)$ centered at points
in a lattice $\Gamma_x=\Gamma_{x,\ep}\subset R$, obtained
by translation.
We call these copies ``cells''.
$R\cong\mathbb{R}^\zeta$ is tiled by copies of the fundamental domains
centered at points in $\Gamma_x$. We call these fundamental
domains ``blocks''. Every block contains exactly one cell.
Let $\alpha=\alpha(\ep)=\mathrm{diam}(R/\Gamma_x)$ 
and $C_1=C_1(\ep)=\mathrm{vol}(R/\Gamma_x)$ be the diameter
and volume of a block. They are independent of $x$.



\begin{lemma}\label{cellcover}
Let $C_2=\mathrm{vol}(D(1))$ and $C_0=C_0(\ep)=C_2\cdot C_1^{-1}$.
If $r_j>\alpha, j=1,\dots, u$, then for every $x$,
the number of cells in $\Phi_x^{-1}(K(\bar r)\cap W(x))$ covered by
$D(\bar r)$ is at least
$$g_-(\bar r)=C_0\cdot\prod_{j=1}^u (r_j-\alpha)^{\zeta_j}.$$
The number of $J$-cells that intersect it is at most
$$g_+(\bar r)=C_0\cdot\prod_{j=1}^u (r_j+\alpha)^{\zeta_j}+1.$$
\end{lemma}

\begin{proof}
As $\alpha$ is the diameter of a block, then if a block intersects
 $D(r_{1}-\alpha, \dots, r_{u}-\alpha)$, then it is fully covered
by $D(\bar r)$.
The number of blocks intersecting $D(\bar r)$ is at least the
volume of itself divided
by the volume of a block. Since every block contains one cell,
the number
of $J$-cells covered by $D(\bar r)$ is at least the number of blocks
covered by it, which is at least $g_-(\bar r)$.

Proof of the other statement is analogous.
\end{proof}



\section{Constructions along Unstable Foliation}


We define
$$V_k^m(\bar r)=\bigcap_{i=0}^{m-1}f^{-ik}(K(\bar r)),
V_k(\bar r)=V_k^\infty(\bar
r),
U_k(\bar r)=U_k(\bar r, \ep)=\bigcap_{i=-\infty}^\infty f^{ik}(K(\bar r,\ep)).$$

For $x\in K(\bar r)$, Let $K(\bar r, x)$ be the connected component of
$K(\bar r)\cap W(x)$ containing $x$.
$K(\bar r,x)=\Phi_{z_x}(D(\bar r))$ for some $z_x\in W(x)$. Let
$$V_k^m(\bar r,x)=V_k^m(\bar r)\cap K(\bar r,x),\;
V_k(\bar r,x)=V_k(\bar r)\cap K(\bar r,x), $$


\begin{proposition}\label{entest}
Let $k$ and $\bar r$ be such that $r_j(\chi_j-\ep)^k>\alpha, j=1,\dots, u$.
Let
\begin{align*}
g_-(k,\bar r, \ep)=&C_0\cdot\prod_{j=1}^u (r_j(\chi_j-\ep)^k-\alpha)^{\zeta_j}.\\
g_+(k,\bar r, \ep)=&C_0\cdot\prod_{j=1}^u (r_j(\chi_j+\ep)^k+\alpha)^{\zeta_j}+1.
\end{align*}
Then for every $x\in K(\bar r)$, $V_k(\bar r,x)$ is a Cantor set and
$$\log g_-(k,\bar r,\ep)\le h(f^{k}, {V_k(\bar r,x)})
\le\log g_+(k,\bar r, \ep).$$
\end{proposition}

\begin{proof}

The idea of the following argument is due to Ma\~n\'e \cite{Ma78}.
Since $K(\bar r,x)=\Phi_{z_x}(D(\bar r))\subset W(x)$, 
by Corollary \ref{excor},
$$f^k(\Phi_{z_x}(D(r_{1},\dots, r_{u})))\supset \Phi_{f^k(z_x)}(D
((\chi_{1}-\ep)^kr_{1}, \dots, (\chi_{u}-\ep)^kr_{u})).$$
By Lemma
\ref{cellcover}, $f^k(K(\bar r,x))$ covers
at least $g_-(k,\bar r, \ep)$ full blocks.
So $f^k(K(\bar r,x))\cap K(\bar r)$ consists
of at least $g_-(k,\bar r, \ep)$ cells. By induction,
$$K(\bar r,x)\cap f^{-k}(K(\bar r))\cap
\dots\cap f^{-mk}(K(\bar r))=V_k^m(\bar r,x)$$ consists of at least
$(g_-(k,\bar r, \ep))^m$ connected components,
each of which is a copy of $f^{-mk}(K(\bar r,x))$. So finally we know
that $V_k(\bar r,x)$ is a Cantor set. For every $m$ and $\delta$ sufficiently
small, every connected component of $V_k^m(\bar r,x)$ intersects $V_k(\bar
r,x)$. From each of such an intersection we can take a point such that these
points form an $(m,\delta)$-separated set, whose cardinality is at least
$(g_-(k,\bar r, \ep))^m$. We have
$$h(f^k, {V_k(\bar r,x)})\ge\log g_-(k,\bar r, \ep).$$

Proof of the other side is analogous.
\end{proof}



Take $\alpha_0>0$ such that $g_-(\alpha_0,\alpha_0,\dots,\alpha_0)>1$. Then

\begin{corollary}\label{nonem}
Let $k$ and $\bar r$ be such that $r_j(\chi_j-\ep)^k\ge\alpha_0, j=1,\dots, u$.
Then
\begin{enumerate}
\item $h(f^k, V_k(\bar r))\ge\log g_-(k,\bar r, \ep)$.
\item $U_k(\bar r)$ is non-empty and $f$-invariant.
\item $h(f^k, U_k(\bar r))\le h(f^k, V_k(\bar r))$.
\end{enumerate}
\end{corollary}


\begin{proposition}
If for every $j$,
\begin{equation*}\label{restri1}
r_{j}>\frac{\alpha_0}{(\chi_{j}-\ep)^{k}}\cdot
\frac{(\chi_{j}-\ep)^{k}+1}{(\chi_{j}-\ep)^{k}-1}
\end{equation*}
then
$h(f^k, U_k(\bar r))\ge\log g_-(k,\bar r, \ep)$.
\end{proposition}

\begin{proof}
Under the assumption, we can find $r_j'$ such that
$$\frac{\alpha_0}{(\chi_{j}-\ep)^k}<r_j'<\frac{(\chi_{j}-\ep)^{k}-1}
{(\chi_{j}-\ep)^{k}+1}\cdot r_{j}$$

Let $\bar r'=(r_1', r_2', \dots, r_u')$ and $\bar\delta=\bar r-\bar r'$.
By Corollary \ref{nonem}, $U_k(\bar r')$ is nonempty.
Take $x_0\in U_k(\bar r')\subset U_k(\bar r)$.
For every $m\in\mathbb{Z}$, $f^{mk}(x_0)\in K(\bar r')$.
We note that $\Phi_{x_0}(D(\bar\delta))\subset K(\bar r,x)$.

For every $m\ge 0$, we have
$$f^{-mk}(\Phi_{x_0}(D(\bar\delta)))\subset
\Phi_{f^{-mk}(x_0)}(D(\bar\delta))\subset K(\bar r,
f^{-mk}(x_0))\subset K(\bar r).$$
So
$\Phi_{x_0}(D(\bar\delta))\subset f^{mk}(K(\bar r))$, and hence
$$\Phi_{x_0}(D(\bar \delta))\cap V_k(\bar r)\subset U_k(\bar r)$$

From the bounds of $r'_j$ we have
$(r_j-r_j')(\chi_{j_0}-\ep)^{k}>r_j+r_j', j=1,2,\dots,u$,
which implies
$$f^k(\Phi_{x_0}(D(\bar\delta)))\supset \Phi_{f^k(x_0)}(D(\bar\delta))\supset
K(\bar r,f^k(x_0)).$$
So finally we obtain
\begin{align*}
h(f^k,U_k(\bar r))\ge& h(f^k, \Phi_{x_0}(D(\bar\delta))\cap V_k(\bar r))\\
=& h(f^k, ( \Phi_{x_0}(D(\bar \delta))\cap K(\bar r,x))\cap f^{-k}(V_k(\bar
r)))\\
=& h(f^k, f^k( \Phi_{x_0}(D(\bar \delta)))\cap V_k(\bar r))\\
\ge&h(f^k, K(\bar r,f^k(x_0))\cap V_k(\bar r))\\
=& h(f^k, V_k(\bar r, f^k(x_0)))\\
\ge& \log g_-(k,\bar r, \ep)
\end{align*}
\end{proof}

\begin{proposition}
$h(f^k, U_k(\bar r))\le \log g_+(k,\bar r, \ep)$.
\end{proposition}

\begin{proof}
As $U_k(\bar r)$ is compact and $f^k$-invariant, it is enough to show that
for every ergodic measure $\mu$ of $f^k$
supported on $U_k(\bar r)$, $h_\mu(f^k)\le\log g_+(k,\bar r, \ep)$.

Define a measurable partition $\xi$ such that $\xi(x)=K(\bar r,x)$.
Then $\xi$ is subordinate to the foliation $W^u$ with respect to $\mu$. Moreover,
$\xi$ is increasing and $f^{-mk}\xi$ generates as $m\to\infty$.
Let $\mu_x$ be the conditional measure associated with $\xi$.
By Proposition \ref{entest} for every $x\in U_k(\bar r)$
(hence for $\mu$-a.e. $x$), $\xi(x)$ is covered (mod $0$)
by at most $g_+(k,\bar r, \ep)$ non-null elements of
$f^{-k}\xi$, denoted by $\eta_j(x), j=1,2,\dots,q(x)$. We have
$$\int_{\xi(x)}-\log\mu_x(f^{-k}\xi)(x)d\mu_x=\sum_{j=1}^{q(x)}
-\mu_x(\eta_j(x))\log\mu_x(\eta_j(x))\le\log q(x)\le\log g_+(k,\bar r, \ep).$$
Apply the result of Ledrappier and Young \cite{LY85}:
\begin{align*}
h_\mu(f^k)=H(f^{-k}\xi|\xi)
=\int-\log\mu_x(f^{-k}\xi)(x)d\mu
\le\log g_+(k,\bar r, \ep).
\end{align*}
\end{proof}

\section{Density of the Estimates}
We make the following assumptions:
\begin{enumerate}
\item $k$ is large enough such that $(\chi_{j}-\ep)^{k}>2$ for every $j$.
\item $r_j(\chi_j-\ep)^k\ge 3\alpha_0$ for every $j$.
\item $\pi_\zeta(Q_\ep(D(\bar r)))\subset(-\frac12,\frac12)^\zeta$.
\end{enumerate}
We take $\tau>0$ such that
$\pi_\zeta(Q_\ep(D(\tau,\tau,\dots,\tau)))\subset(-\frac12,\frac12)^\zeta$.
For $k$ large enough, let
$$\Omega_k=\Omega_{k,\ep}=\{\bar r\in\mathbb{R}^u| \frac{3\alpha_0}{(\chi_j-\ep)^k}\le r_j\le\tau\}.$$
Then every $\bar r\in\Omega_k$ satisfies the assumptions.

\begin{corollary}\label{uentes}
Under the assumptions, we have
$$\log g_-(k,\bar r, \ep)\le h(f^k, U_k(\bar r, \ep))\le
\log g_+(k,\bar r, \ep)$$
\end{corollary}

\begin{lemma}\label{erres}
For every $\gamma>0$, there is $\ep_\gamma>0$ and $k_\gamma$ such that if
$0<\ep\le\ep_\gamma$ and $k\ge k_\gamma$, then for every $\bar r\in\Omega_k$, $\log g_+(k,\bar r, \ep)-\log g_-(k,\bar r, \ep)<\gamma k$.
\end{lemma}

\begin{proof}
Note that
\begin{align*}
\psi_k(\bar r, \ep)=\log g_+(k,\bar r, \ep)-\log g_-(k,\bar r, \ep)=\sum_{j=1}^u\log\frac{(r_j(\chi_j+\ep)^k+\alpha)^{\zeta_j}}{(r_j(\chi_j-\ep)^k-\alpha)^{\zeta_j}}
\end{align*}
is a continuous function of $\bar r\in\Omega_k$ and is decreasing for
each $r_j$. So
\begin{align*}
\frac1k\psi_k(\bar r, \ep)&\le\frac1k\psi_k(\frac{3\alpha_0}{(x_1-\ep)^k},\dots,\frac{3\alpha_0}{(x_u-\ep)^k}):=\psi_k(\ep)\\
&=\frac1k(\log(C_0\cdot\prod_{j=1}^u(\frac{3\alpha_0(\chi_j+\ep)^k}{(\chi_j-\ep)^k}+\alpha)^{\zeta_j}+1)-\log
C_0-\sum_{j=1}^u\log(3\alpha_0-\alpha)^{\zeta_j})\\
&\to\sum_{j=1}^u\zeta_j\log\frac{\chi_j+\ep}{\chi_j-\ep} \text{ as }k\to\infty\\
&\to 0 \text{ as }\ep\to 0.
\end{align*}
The result follows.

\end{proof}

Proposition \ref{promain}, hence our main theorem, is a trivial corollary
of the following one, by taking 
$$M=\bigcup_{j=1}^k f^j(U_k(\bar r,\ep))$$
for $\beta_1, \beta_2, k, \bar r,\ep$ as in the proposition.

\begin{proposition}\label{fineest}
For every $\beta_1, \beta_2\in[0, h(f)]$, $\beta_1<\beta_2$, there is $\ep_0$
and $k_0$
such that for $0<\ep\le\ep_0$ and $k\ge k_0$, there is $\bar r\in\Omega_{k,\ep}$ such that $h(f^k, U_k(\bar
r,\ep))\in[\log g_-(k,\bar r, \ep), \log g_+(k,\bar r, \ep)]\subset[k\beta_1,k\beta_2]$.
\end{proposition}

\begin{proof}
$g_-(k,\bar r)$ is continuous in $\bar r$ and increasing for each $r_j$.
We have
\begin{align*}
\frac1k\min\{\log g_-(k,\bar r,\ep)|\bar r\in\Omega_{k,\ep}\}&=\frac1k\log g_-(k,(\frac{3\alpha_0}{(x_1-\ep)^k},\dots,\frac{3\alpha_0}{(x_u-\ep)^k}),\ep)\\
&=\frac1k(\log C_0(\ep)+\sum_{j=1}^u\log(3\alpha_0-\alpha)^{\zeta_j})\\
&\to 0 \text{ as }k\to\infty,\\
\frac1k\max\{\log g_-(k,\bar r,\ep)|\bar r\in\Omega_{k,\ep}\}&=\frac1k\log g_-(k,(\tau,\cdots,\tau),\ep)\\
&=\frac1k(\log C_0(\ep)+\sum_{j=1}^u\log(\tau(\chi_j-\ep)^k-\alpha)^{\zeta_j})\\
&\to\sum_{j=1}^u\zeta_j\log(\chi_j-\ep)\text{ as }k\to\infty\\
&\to h(f)\text{ as }\ep\to 0.
\end{align*}
From Proposition \ref{erres}, there is $\ep_0>0$ and $k_0$ such that for
$0<\ep<\ep_0$ and $k\ge k_0$, $\psi_k(\ep)<k(\beta_2-\beta_1)$ and $\frac1k\log
g_-(k,\Omega_{k,\ep},\ep)\supset[\beta_1, \beta_2]$. Hence there is $\bar r\in\Omega_{k,\ep}$
such that $\log g_-(k,\bar r,\ep)=k\beta_1$. By Corollary \ref{uentes},
$h(f^k, U_k(\bar r, \ep))\in[\log g_-(k,\bar r, \ep), \log g_+(k,\bar r, \ep)]\subset[k\beta_1,k\beta_2]$. 
\end{proof}

\section*{Acknowledgments}
We would like to thank Anatole Katok for numerous discussions and explanations.
We also thank Douglas Lind and Klaus Schmidt for their knowledge on the topic.



\end{document}